\documentclass[11pt]{article}
\usepackage{amssymb,amsfonts,amsmath,latexsym,epsf,tikz,url}
\usepackage[usenames,dvipsnames]{pstricks}
\usepackage{pstricks-add}
\usepackage{epsfig}
\usepackage{pst-grad} % For gradients
\usepackage{pst-plot} % For axes
\usepackage[space]{grffile} % For spaces in paths
\usepackage{etoolbox} % For spaces in paths
\usepackage{float}
\usepackage{soul}
\usepackage{tikz}
\usepackage[colorinlistoftodos]{todonotes}
\usepackage{pgfplots}
\usepackage{mathrsfs}
\usepackage[colorlinks]{hyperref}
\usetikzlibrary{arrows}
\makeatletter % For spaces in paths
\patchcmd\Gread@eps{\@inputcheck#1 }{\@inputcheck"#1"\relax}{}{}
\makeatother

\newtheorem{theorem}{Theorem}[section]

\newtheorem{proposition}[theorem]{Proposition}
\newtheorem{observation}[theorem]{Observation}
\newtheorem{corollary}[theorem]{Corollary}
\newtheorem{lemma}[theorem]{Lemma}

\newtheorem{definition}[theorem]{Definition}

\newcommand{\qed}{\hfill $\square$\medskip}

\begin{document}

\def\nt{\noindent}

\title{Introduction to double coalitions in graphs}

\author{H. Golmohammadi}

%\date{\today}

\maketitle

\begin{center}
Novosibirsk State University, Pirogova str. 2, Novosibirsk, 630090, Russia\\ 

\medskip
Sobolev Institute of Mathematics, Ak. Koptyug av. 4, Novosibirsk, 630090, Russia\\

{\tt h.golmohammadi@g.nsu.ru }
\end{center}

%%%%%%%%%%%%%%ABSTRACT%%%%%%%%%%%%%%%%%%%%%%%%%%%%%%%%%%%%%%%%%%%%%%%%%%%%%%%%%%%%
\begin{abstract}
 Let $G(V, E)$ be a finite, simple, isolate-free graph. A set $D$ of vertices of a graph $G$ with the vertex set $V$ is a double dominating set of $G$, if every vertex 
 $v\in D$ has at least one neighbor in $D$ and every vertex $v \in V \setminus D$ has at least two neighbors in $D$.
  A double coalition consists of two disjoint sets of vertices $V_{1}$ and $V_{2}$, neither of which is a double dominating set 
  but their union $V_{1}\cup V_{2}$  is a double dominating set.
 A double coalition partition of a graph $G$ is a partition $\Pi = \{V_1, V_2,..., V_k \}$ of $V$ such that no subset of $\Pi$ is a double dominating set of $G$, but for every set $V_i \in \Pi$, there exists a set $V_j \in \Pi$ such that $V_i$ and $V_j$ form a double coalition. In this paper, we study properties of double coalitions in graphs.
\end{abstract}

  
\medskip
\noindent{\bf AMS Subj.\ Class.:}  05C60. 

%%%%%%%%%%%%%%%%%%%%%%%%%%%%%%%%%%%%%%%%%%%%%%%%%%%%%%%%%%%%%%%%%%%%%%%%%%%%%%%%%
%%%%%%%%%%%%%%%%%%%%%%%%%%%%%%%%%%%%%%%%%%%%%%%%%%%%%%%%%%%%%%%%%%%%%%%%%%%%%%%%%

\section{Introduction} 

We only consider finite, simple and undirected graphs $G$ of order $n$ with vertex set $V(G)$. The open neighborhood $N_G(v)$ of a vertex $v$ in $G$ is the set of vertices adjacent to $v$, while the closed
neighborhood of $v$ is the set $N_G[v]=\{v\}\cup N_G(v)$. Each vertex of $N(v)$ is called a neighbor
of $v$, and the cardinality of $|N(v)|$ is called the degree of $v$, denoted by $deg(v)$. An isolated vertex is
a vertex of degree 0.
The minimum and maximum degree of graph vertices are denoted by $\delta(G)$ and $\Delta(G)$, respectively.
Domination is among very popular topics in graph theory. Various aspects of domination have been surveyed in the books \cite{16,17}.
A set $S \subseteq V$ is called a dominating set if every vertex of $ V \setminus S$ is adjacent to at least one vertex in $S$. The domination
number $\gamma(G)$ is the cardinality of the minimum size dominating set in $G$. Harary and Haynes explored the concept of  
double domination in graphs and, more generally, the concept of $k$-tuple domination in \cite{10}. A subset $D$ of $V$ is a $k$-tuple dominating set of
$G$ if for every vertex $v\in V$, $|N[v] \cap D|\ge k$, that is, $v$ is in $D$ and has at least $k-1$ neighbors in $D$ or $v$ is in $V-D$ and has at
least $k$ neighbors in $D$. The $k$-tuple
domination number of $G$, denoted by $\gamma_{\times k}(G)$, is the minimum cardinality among all $k$-tuple dominating sets of $G$. A $k$-tuple dominating set where $k=2$ is called a double dominating set. In such a case, $\gamma_{\times 2}(G)$ indicates the double domination number of graph $G$. We remark that this parameter is also denoted by $dd(G)$. The double domination in graphs has been well researched in \cite{6,7,8,9}. A $k$-tuple domatic partition is a partition
of vertices of a graph into $k$-tuple dominating sets. The maximum cardinality of a $k$-tuple domatic partition is called the $k$-tuple domatic number, denoted by $d_{\times k}(G)$. The $k$-tuple domatic number of a graph was introduced by Harary and Haynes in 1998 \cite{11}, and has studied further in \cite{18}. It can be seen that the case $k=2$ is known as the 2-tuple domatic number, and denoted by $d_{\times 2}(G)$. 

In 2020, Haynes et al. \cite{12} introduced the coalitions in graphs. Let $V_1\subseteq V$ and $V_2\subseteq V$ denote two disjoint subsets of $V$. They form a coalition if none of them are dominating sets, but their union $V_1\cup V_2$ is.
A coalition partition is a vertex partition $\mathcal{P}=\{V_1,V_2,\dots,V_k\}$ of $V$ such that for every $i\in\{1,2,\dots,k\}$ the set $V_i$ is either a dominating set and $|V_i|=1$, or there exists another set $V_j$ so that they form a coalition. The maximum cardinality of a coalition partition is called the {coalition number of the graph, and denoted by $C(G)$.
Coalition and its variations in graphs have been extensively studied in \cite{1,2,3,4,5,13,14,15}. 
In order to develop future research, we propose a new perspective on coalitions involving double dominating sets in graphs.

\medskip 
 We proceed as follows. In Section 2, we introduce the double
 coalitions in graphs and establish some bounds on the double coalition number. In Section 3, we determine the double coalition
number for several classes of graphs.

\section{Existence and preliminary results }

We start this section by proving two definitions.

 \begin{definition}[double coalition]
Two sets $V_1,V_2\subseteq V(G)$ form a double coalition in a graph $G$ if they are not double dominating sets but their union is a double dominating set in $G$. 
 \end{definition} 
 
\begin{definition}[double coalition partition]\label{2.2} 
A double coalition partition of a graph $G$ is a partition $\Pi=\{V_1,V_2,\dots,V_k\}$ of the vertex set $V$ such that any $V_i\in \Pi, 1\leq i \leq k,$ is not a double dominating set but forms a double coalition with another set $V_j \in \Pi$ that is not a double dominating set. The maximum cardinality of a double coalition partition is called the double coalition number of the graph 
and denoted by $DC(G)$.
A $dc$-partition of $G$ having $DC(G)$ subsets is called a $DC(G)$-partition.
\end{definition}

Now we call up the known result that we need for what follows.

\begin{theorem} ~\label{1}\cite{10}
    Every graph $G$ with no isolated has a double dominating set and hence a double domination number.
\end{theorem}

As an immediate outcome of Theorem \ref{1}, we have the following trivial observation.
\begin{observation} 

If a graph $G$ contains an isolated vertex, then $DC(G)=0$.
    
\end{observation}

 The primary objective of this section is to show that an isolate-free graph $G$ has a $dc$-partition. 

\begin{theorem}\label{e1}
Every graph $G$ with no isolated vertex has a $dc$-partition.

\end{theorem} 

\begin{proof} 
	 Let $G$ has a double domatic partition $\mathcal{D}=\{D_1, D_2,\ldots, D_k\}$ with $d_{\times 2}(G)=k$. It is clear that $|D_i|\geq 2$. In the following, we demonstrate the process of constructing a $dc$-partition $\Pi$ of $G$. For any integer $1\leq i\leq k-1$, if $D_i$ is not a minimal double dominating set, then there is a set $D'_i\subset D_i$ which is a minimal double dominating set. So, we swap $D_i$ by $D'_i$ in $\cal D$, and add $D_i\setminus D'_i$ to $D_k$. Therefore, we may assume that all sets $D_i\in {\cal D}$ with $1\leq i\leq k-1$ are minimal double dominating sets of $G$. Now, we establish a $dc$-partition $\Pi$ for $G$. We begin with $\Pi=\emptyset$. Now, we divide each set $D_i\in {\cal D}$ with  $1\leq i\leq k-1$ into two nonempty sets $D^1_i$ and $D^2_i$, and then, we add $D^1_i$ and $D^2_i$ to $\Pi$. It is clear that neither $D^1_i$ nor $D^2_i$ is a double dominating set but $D^1_i\cup D^2_i$ is a double dominating set. Then, clearly $|\Pi|\geq 2k-2$. Next, consider the set $D_k\in {\cal D}$. If $D_k$ is a minimal double dominating set, then we split $D_k$ into two nonempty sets $D^1_k$ and $D^2_k$, and then, we add $D^1_k$ and $D^2_k$ to $\Pi$. So, this creates a $dc$-partition $\Pi$ in $G$ with $|\Pi|=2k$. Next, suppose that $D_k$ is not a minimal double dominating set. Then, there is a set $D'_k\subset D_k$ which is a minimal double dominating set. Now, we split $D'_k$ into two sets $D'^1_k$ and $D'^2_k$, and append them to $\Pi$. Now, consider the set $D''_k:=D_k\setminus D'_k$. Since $\cal D$ is a double domatic partition of $G$ with the maximum cardinality $d_{\times 2}(G)=k$, the set $D''_k$ is not a double dominating set. If $D''_k$ forms a double coalition with one of the sets of $\Pi$, then we add $D''_k$ to $\Pi$, and therefore $|\Pi|=2k+1$. If $D''_k$ does not form a double coalition with any set of $\Pi$, then we remove $D'^2_k$ from $\Pi$ and append $D'^2_k\cup D''_k$ to $\Pi$, and  consequently $|\Pi|=2k$, which completes the proof. \qed

 \end{proof}

The proof of Theorem \ref{e1} allows us to deduce the following corollary.

\begin{corollary}
   If a graph $G$ has $\delta(G)\geq 1$, then $DC(G)\geq 2d_{\times 2}(G)$.
\end{corollary}

Harary and Haynes \cite{10} mentioned that for all graphs $G$ with $\delta(G)\geq k-1$, $d_{\times k}(G)\geq 1$. So, it is obvious that for $k=2$, $d_{\times 2}(G)\geq 1$.\medskip

Hence as a consequence of the statement above and Theorem \ref{e1}, we derive the following result.

\begin{corollary}~\label{cor1}
If $G$ is a graph of order $n$ with $\delta(G)\geq 1$, then $2\leq DC(G)\leq n$.
\end{corollary}

The upper bound of Corollary~\ref{cor1} is achieved  by complete graphs $K_n$.\medskip

We next present a technical key lemma, which permits us to estimate the number of double coalitions involving any set in a $DC(G)$-partition of $G$. 

\begin{lemma}\label{thm:Delta1}
	If $G$ is an isolate-free graph with maximum degree $\Delta(G)$, then for any $DC(G)$-partition $\Pi$ and for any $X\in \Pi$, the number of double coalitions formed by $X$ is at most  $\Delta(G)$.	    
 
\end{lemma} 
\begin{proof} Since $X\in\Pi$, $X$ is not a double dominating set. So, there is a vertex $w$ such that $|N[w] \cap X|\leq 1$. We now distinguish two cases.

\nt {\bf Case 1.} Let $|N[w] \cap X|=1$. We first assume that $w\in X$. If a set $A\in \Pi$ forms a double coalition with $X$, then  $A\cup X$ is a double dominating set of $G$.
Since $X\cap N(w)=\emptyset$, we must have $A\cap N(w)\neq\emptyset$. Thus, there are at most $|N(w)|$ sets such as $A$ that can form a double coalition with $X$, and consequently, $X$ is in at most $\Delta(G)$ double coalitions. Now, suppose that $w\not\in X$ and $X$ contains exactly one of the members of $N(w)$. We may assume that there is a set $U\in \Pi$ such that $w\in U$. Without loss of generality, let $U$ and $X$ do not form a double coalition, but each of $Y_1, Y_2,\ldots, Y_k$ for $1\leq i\leq k$ is in a double coalition with $X$. Therefore, we must have $Y_i\cap N(w)\ne \emptyset$ for $1\leq i\leq k$ such that $k\leq |N(w)|-1$. Hence, there at most $|N(w)|-1\leq \Delta(G)-1$ sets that can be in a double coalition with $X$. Note that in the case $U$ and $X$ form a double coalition, $X$ is in at most $\Delta(G)$ double coalitions.

\nt {\bf Case 2.} Let $|N[w] \cap X|=0$. It follows that $w\not\in X$ and $X\cap N(w)=\emptyset$. We may assume that $U$ and $X$ do not form a double coalition, but each of
$Y_1, Y_2,\ldots, Y_k$ for $1\leq i\leq k$ is in a double coalition with $X$. So, we must have $Y_i\cap N(w)\ne \emptyset$ for $1\leq i\leq k$ such that $k\leq \frac{|N(w)|}{2}$. Thus, there are at most $\frac{|N(w)|}{2} \leq \frac{\Delta(G)}{2}$ sets that can be in a double coalition with $X$. 

According to the two cases above, we conclude that $X$ is in at most $\Delta(G)$ double coalitions.\qed

\end{proof}

    We establish next an upper bound on the double coalition in terms of the maximum degree of $G$.

\begin{theorem}\label{thm:Delta2}
	Let $G$ be an isolate-free graph wit $\delta(G)=1$. Then, 	$DC(G) \leq \Delta(G)+1$.
\end{theorem} 
\begin{proof} 
Let $x$ be a vertex of $G$ with $\deg(x)=1$. Suppose that $\Pi$ be a $DC$-partition of $G$. Let $X \in \Pi$ such that $x\in X$. If $N(x)\subseteq X$, then any set of $\Pi \setminus X$ must form a double coalition only with $X$. So, by Lemma~\ref{thm:Delta1}, $DC(G)\leq \Delta(G)+1$. Now, assume that $N(x)\nsubseteq X$. Let $Y\neq X$ and $Z\neq X$ be two sets of $\Pi$. If $Y$ and $Z$ forms a double coalition, then $Y\cup Z$ is a double dominating set. Since $x\not\in Y\cup Z$, $x$ must have at least two neighbors in $Y\cup Z$, then we reach a contradiction because $\deg(x)=1$. Therefore, every set of $\Pi$ must only form a double coalition with $X$. Hence, by Lemma~\ref{thm:Delta1}, $DC(G)\leq \Delta(G)+1$, as desired. \qed
\end{proof}

Applying Theorem \ref{thm:Delta2} to the class of trees, we get the following result.

\begin{corollary}
If $T$ is a tree of order $n$, then $DC(T) \leq \Delta(T)+1$.

\end{corollary}

\section{Double coalition number for special graphs }

In this section, we deal with the problem of obtaining the exact value of the double coalition number. We first recall the following fact.

\begin{observation}\label{3} \cite{10}
    For $r, s \geq 3$, we have $\gamma_{\times 2}(K_{r,s})=4$.

\end{observation}

The following proposition gives us the double coalition number of complete bipartite
graph.

\begin{proposition}
If $G=K_{r,s}$ is a complete bipartite graph and $r\geq s \geq 3$, then $DC(K_{r,s})=r+s-2=n-2$.
\end{proposition}
\begin{proof} Let $G=K_{r,s}$ be a complete bipartite graph with two partite sets $A=\{v_1,v_2, \dots,v_s\}$ and $B=\{u_1,u_2,\dots,u_r\}$  such that $|A|\leq|B|$. One can observe that the vertex partition  $\Pi=\{{\{v_1,v_2,u_{1}\},\{v_3\},\dots,\{v_{s-1}\},\{v_s\}, \{u_2\},\dots,\{u_{r-1}\},\{u_r\} }\}$ is a $dc$-partition
of $G$ of order $r+s-2$,  Thus, $DC(K_{r,s})\geq r+s-2$. Next, we shall show that $DC(K_{r,s})\leq r+s-2$. We proceed further with the following cases.

\nt {\bf Case 1.} Assume that $\Pi$ is a $dc$-partition of $G$ of order $r+s=n$ such that $|Vi|=1$ for $1 \leq i \leq r+s$. By Observation \ref{3}, we deduce that no two sets of $\Pi$ can form a double coalition of order 4. Hence, $\Pi$ does not exist.

\nt {\bf Case 1.} Let $\Pi$ is a $dc$-partition of $G$ of order $r+s-1=n-1$. It follows that $\Pi$ consists of a set of cardinality 2 and $r+s-3$ singleton sets. Since $\gamma_{\times 2}(K_{r,s})=4$, no two sets of $\Pi$ can be in a double coalition. Thus, $\Pi$ does not exist.\medskip

In light of the above, we conclude that $DC(K_{r,s})\leq r+s-2$. Therefore, $DC(K_{r,s})=r+s-2=n-2$.\qed

\end{proof}

 Next we determine the double coalition number of all paths and cycles.

\begin{theorem}
    
	For any path $P_n$,
		\begin{equation*}
		DC(P_n)=\left\{
		\begin{aligned}[c]
		2, & { \ \ if\ } 2\leq n\leq5 \\
		3, & { \ \ if \ } n\geq 6 \\
		\end{aligned}\right.
		\end{equation*}
	\end{theorem}
	\begin{proof}
		Assume that $\Pi=\{V_1, V_2,\dots,V_n\}$ is a $dc$-partition of $P_n$. By Theorem \ref{thm:Delta2} and Corollary \ref{cor1}, we have $2\leq DC(P_n)\leq 3$ for any path $P_n$. It is easy to verify that for $n=2$, $DC(P_2)=2$. Now, we proceed to show that $DC(P_n) \ne 3$ where $3 \leq n\leq 5$. We consider four possible cases.

\nt {\bf Case 1.} $\Pi$ consists of three singleton sets. Since $\gamma_{\times 2}(P_3)=3$, no two singleton sets are $dc$-partners. Hence, $DC(P_3) \ne 3$.

\nt {\bf Case 2.} $\Pi$ consists of a set of cardinality 2 and two singleton sets. Since $\gamma_{\times 2}(P_4)=4$, no two sets are in $dc$-partners. Thus, $DC(P_4) \ne 3$.

\nt {\bf Case 3.} $\Pi$ consists of a singleton set, say $A$, and  two sets of cardinality 2 such as $B$ and $C$. Since $\gamma_{\times 2}(P_5)=4$, it can be seen that neither $B$ nor $C$ is able to form a double coalition with $A$. Then, $DC(P_5) \ne 3$.

\nt {\bf Case 4.} $\Pi$ consists of a set of cardinality 3,  and two singleton sets. We note that each
singleton set  must be a $dc$-partner of a set of cardinality 3, which is impossible, as $P_5$ has
a unique double dominating set of cardinality 4. Hence, $DC(P_5) \ne 3$.\medskip

Based on the analysis of all the above cases, we infer that $DC(P_n)=2$ where $3 \leq n\leq 5$.\medskip

Finally, let $n\geq 6$. By Theorem \ref{thm:Delta2}, for any path $P_n$ we have $DC(P_{n})\leq 3$. Now, we find a maximal double coalition partition of order 3 for $P_n$ where $n\ge 6$.

		$$\Pi(P_n)= \left\{V_1=\left\{V\backslash\{v_3,v_4\}\right\}, V_2=\{v_3\}, V_3=\{v_4\}\right\}.$$
  
  One can observe that $V_2$ and $V_3$ are $dc$-partners of $V_1$. Therefore, the proof is complete.\qed

	\end{proof}

Before presenting the next result, we remember the following fact.

\begin{observation}\label{4} \cite{10}
   For any cycle $C_n$, $\gamma_{\times 2}(C_n)=\lceil\frac{2n}{3}\rceil$.
\end{observation}

\begin{theorem}
    For a cycle $C_n$ of order $n\geq 3$, $DC(C_n)=3$.
\end{theorem}
\begin{proof}
Let $\Pi=\{V_1, V_2,\dots,V_n\}$ be a $dc$-partition of cycle $C_n=(v_1,v_2,\dots,v_{n},v_1)$. We first show that $DC(C_n)\neq 4$. 
Suppose, to the contrary, that $DC(C_n)=4$. Let $\Pi=\{A, B, C, D\}$ be a $DC(C_n)$-partition.  
By Lemma \ref{thm:Delta1}, each set of $\Pi$ is in double coalition with at most two sets of $\Pi$.
So, without loss of generality, assume that $A$ and $B$ form a double coalition, and $C$ and $D$ form a double coalition. Since $\gamma_{\times 2}(C_n)=\lceil\frac{2n}{3}\rceil$, it holds that $|A|+|B|\geq \lceil\frac{2n}{3}\rceil$ and $|C|+|D|\geq \lceil\frac{2n}{3}\rceil$. Therefore, $|A|+|B|+|C|+|D|\geq 2\lceil\frac{2n}{3}\rceil$. On the other hand, we know that $|A|+|B|+|C|+|D|=n$. Then $n\geq 2\lceil\frac{2n}{3}\rceil$, which is a contradiction. Hence, $DC(C_n)\neq 4$ and $DC(C_n)\leq 3$. 
Now, we establish a maximal double coalition partition of order 3 for $C_n$ as follows.

 $$\Pi(C_n)= \left\{V_1=\left\{V\backslash\{v_{n-1},v_n\}\right\}, V_2=\{v_{n-1}\}, V_3=\{v_n\}\right\}.$$
 
Note that each of $V_2$ and $V_3$ form a double coalition with $V_1$.  This completes the proof. \qed

\end{proof}


\medskip


\begin{thebibliography}{99}

\bibitem{1}  S. Alikhani, D. Bakhshesh, H.Golmohammadi, Total coalitions in graphs, Quaest.Math.
(2024): 1–12. https://doi.org/10.2989/16073606.2024.2365365.

 

\bibitem{2}  S. Alikhani, D. Bakhshesh, H. Golmohammadi, S. Klavzar, On independent coalition in graphs and independent coalition graphs, Discuss. Math. Graph Theory
(2024) doi.org/10.7151/dmgt.2543.



\bibitem{3}  S. Alikhani, D. Bakhshesh, H. Golmohammadi, E.V. Konstantinova, Connected
coalitions in graphs, Discuss. Math. Graph Theory (2023) doi.org/10.7151/
dmgt.2509.


\bibitem{4} D. Bakhshesh, M.A. Henning and D. Pradhan, On the coalition number of trees, Bull. Malays.
Math. Sci. Soc. 46 (2023) 95.

\bibitem{5}J. Bar\'at, Z.L. Bl\'azsik, General sharp upper bounds on the total coalition number, Discuss. Math. Graph Theory (2023) doi.org/10.7151/dmgt.2511.


\bibitem{6} A. Cabrera Mart\'{i}nez, New bounds on the double domination number of trees. Discrete Appl. Math. 315 (2022) 97–103.

\bibitem{7} A. Cabrera Mart\'{i}nez and J. A. Rodr\'{i}guez-Vel\'{a}zquez, A note on double domination in graphs. Discrete Appl. Math. 300 (2021)
107–111.

\bibitem{8} A. Cabrera-Mart\'{i}nez, A. Estrada-Moreno, Double domination
in rooted product graphs, Discrete Appl. Mat.
339 (2023)  127–135.

\bibitem{9} J. Harant and M.A. Henning, On double domination in graphs. Discuss. Math. Graph Theory 25 (2005) 29–34.


\bibitem{10} F. Harary, T.W. Haynes, Double domination in graphs, Ars Combin. 55 (2000) 201–213.

\bibitem{11} F. Harary, T.W. Haynes, The k-tuple domatic number of a graph, Math. Slovaca
48 (1998), 161–166.


	\bibitem{12} T.W. Haynes, J.T. Hedetniemi, S.T. Hedetniemi, A.A. McRae and R. Mohan,
	Introduction to coalitions in graphs, AKCE Int. J. Graphs Combin. 17 (2) (2020),
	653–659.
	
	\bibitem{13} T.W. Haynes, J.T. Hedetniemi, S.T. Hedetniemi, A.A. McRae, R. Mohan,
Coalition graphs of paths, cycles, and trees, Discuss. Math. Graph Theory
43 (2023) 931–946.

	
	\bibitem{14} T.W. Haynes, J.T. Hedetniemi, S.T. Hedetniemi, A.A. McRae and R. Mohan, Upper bounds on the coalition number, Austral. J. Combin. 80 (3) (2021),
	442–453.


  \bibitem{15} T.W. Haynes, J.T. Hedetniemi, S.T. Hedetniemi, A.A. McRae, R. Mohan, Coalition
graphs, Comm. Combin. Optim. 8 (2023), no. 2, 423–430.



 \bibitem{16}T. W. Haynes, S. T. Hedetniemi, and M. A. Henning, Domination in Graphs: Core Concepts
Series: Springer Monographs in Mathematics, Springer, Cham, 2023. xx + 644 pp.

	\bibitem{17} T.W. Haynes, S.T. Hedetniemi, P.J. Slater, Fundamentals of Domination in Graphs, in: Chapman and Hall/CRC Pure and Applied Mathematics Series,
	Marcel Dekker, Inc. New York, 1998.
	
	

\bibitem{18} L. Volkmann, Bounds on the k-tuple domatic number of a graph, Math. Slovaca 61 (2011) 851–858.





\end{thebibliography}
\end{document}